\documentclass[11pt]{amsart}

\usepackage{amsfonts}
\usepackage{amsthm, amssymb, ulem, amscd} 
\usepackage{amsmath}
\usepackage[mathscr]{euscript}
\usepackage{color}
\usepackage{enumerate}
\usepackage{graphicx}

\def\crn#1#2{{\vcenter{\vbox{
        \hbox{\kern#2pt \vrule width.#2pt height#1pt
           }
          \hrule height.#2pt}}}}
\def\intprod{\mathchoice\crn54\crn54\crn{3.75}3\crn{2.5}2}
\def\into{\mathbin{\intprod}}

\newcommand{\pa}{\partial}
\newcommand{\na}{\nabla}

\newcommand{\Sym}{\operatorname{Sym}}
\newcommand{\Rm}{\operatorname{Rm}}

\newcommand{\Span}{\operatorname{span}}

\newcommand{\R}{\mathbb R}

\newcommand{\C}{\mathbb C}

\newcommand{\T}{\mathbb T}

\newcommand{\al}{\alpha}

\newcommand{\ga}{\gamma}
\newcommand{\si}{\sigma}
\newcommand{\de}{\delta}
\newcommand{\be}{\beta}

\newcommand{\om}{\omega}
\renewcommand{\th}{\theta}

\newcommand{\Rb}{\overline{R}}

\newcommand{\wh}{\widehat}

\newcommand{\Rt}{\widetilde{R}}
\newcommand{\Lt}{\widetilde{L}}

\newcommand{\Tt}{\widetilde{T}}

\newcommand{\nb}{\overline{\nabla}}

\newcommand{\Om}{\Omega}
\newcommand{\Si}{\Sigma} 
\newcommand{\cL}{\mathcal{L}}

\newcommand{\cE}{\mathcal{E}}

\newcommand{\cI}{\mathcal I}  
\newcommand{\cJ}{\mathcal J}  

\newcommand{\cG}{\mathcal{G}}
\newcommand{\cX}{\mathcal{X}}
\newcommand{\cR}{\mathcal{R}}

\newcommand{\cU}{\mathcal{U}}

\newcommand{\cT}{\mathcal{T}}

\newcommand{\bR}{\mathbf{R}}
\newcommand{\bT}{\mathbf{T}}
\newcommand{\bG}{\mathbf{G}}

\renewcommand{\sp}{\mathsf{p}}

\def\sideremark#1{\ifvmode\leavevmode\fi\vadjust{\vbox to0pt{\vss
 \hbox to 0pt{\hskip\hsize\hskip1em
 \vbox{\hsize2cm\tiny\raggedright\pretolerance10000
  \noindent #1\hfill}\hss}\vbox to8pt{\vfil}\vss}}}

\theoremstyle{plain}
\newtheorem{theorem}{Theorem}[section]
\newtheorem{lemma}[theorem]{Lemma}
\newtheorem{proposition}[theorem]{Proposition}

\theoremstyle{definition}

\newtheorem{definition}[theorem]{Definition}

\theoremstyle{remark}

\numberwithin{equation}{section}

\title[Natural Tensors for Riemannian Submanifolds]{Geodesic
    Normal Coordinates and Natural Tensors for Pseudo-Riemannian Submanifolds}

\author{C Robin Graham} 
\address{Department of Mathematics, University of Washington,
Box 354350\\
Seattle, WA 98195-4350, USA}
\email{robin@math.washington.edu}

\author{Tzu-Mo Kuo}
\address{Department of Mathematics, National Taiwan University\\
Taipei 10617, Taiwan}
\email{tzumokuo@ntu.edu.tw}

\begin{document}

\begin{abstract}
We construct a version of geodesic normal coordinates adapted to a
submanifold of a pseudo-Riemannian manifold and show that the Taylor
coefficients of the metric in these coordinates can be expressed as
universal polynomials in the components of the covariant derivatives of the
background curvature tensor and the covariant derivatives of the second
fundamental form.  We formulate a definition of natural
submanifold tensors and show that these are linear combinations of
contractions of covariant derivatives of the background curvature tensor
and covariant derivatives of the second fundamental form.  We also describe
how this result gives a similar characterization of natural submanifold  
differential operators.  
\end{abstract}

\maketitle

\thispagestyle{empty}

\section{Introduction}\label{intro}
A fundamental result in classical Riemannian geometry is that at the origin
in geodesic normal coordinates, each partial derivative of the metric can
be written as a universal polynomial in the components of the Riemann
curvature tensor and its covariant derivatives.  See, for example,
Corollary 2.9 of \cite{G} for a computation of the first few derivatives,
and its introduction for discussion and references to the
classical literature.  This result found application in the heat equation
proof of the index theorem in \cite{ABP}, where it was used to show that 
the heat kernel coefficients could be written in terms of curvature.  It 
is used in combination with Weyl's classical invariant theory to
characterize natural tensors on Riemannian manifolds as linear combinations
of contractions of
tensor products of the curvature tensor and its covariant derivatives.  In
this paper, we prove an analogous result for submanifolds of
pseudo-Riemannian manifolds and apply it to give an analogous
characterization of natural submanifold tensors.

Let $(M^n,g)$ be a pseudo-Riemannian manifold and let $\Si^k\subset M^n$ be
an embedded submanifold, $1\leq k \leq n-1$.  We assume throughout that
$\Si$ is nondegenerate in the sense that the pullback metric $g|_{T\Si}$ is
nondegenerate.  Denote by
$(p,q)$, $p+q=k$, the signature of $g|_{T\Si}$ and by $(p',q')$,
$p'+q'=n-k$, the signature of $g|_{N\Si}$, where $N\Si$ denotes the normal
bundle of $\Si$.  We will say that such a $\Si$ is of type
$((p,q),(p',q'))$.  Decompose $\R^n = \R^k\oplus \R^{n-k}$ and fix a  
reference quadratic form $h\in S^2\R^n{}^*$ such that $h|_{\R^k}$ has
signature $(p,q)$, $h|_{\R^{n-k}}$ has signature $(p',q')$, and
$\R^k \perp \R^{n-k}$.
Let $\sp\in \Si$ and let $\{e_{\al}\}_{1\leq \al\leq k}$,
$\{e_{\al'}\}_{k+1\leq \al'\leq n}$ be frames for 
$T_\sp\Si$, $N_\sp\Si$, resp., satisfying $g(e_{\al},e_{\be})=h_{\al\be}$, 
$g(e_{\al'},e_{\be'})=h_{\al'\be'}$.  In \S\ref{normalcoordinates} we
construct a submanifold geodesic normal coordinate system
$(x,u)=(x^1,\cdots,x^k,u^{k+1},\cdots,u^{n})$ near $\sp$
depending only on the frames $\{e_{\al}\}$, $\{e_{\al'}\}$.  Such
coordinates are therefore determined up to $H:=O(p,q)\times O(p',q')$.  
These coordinates are constructed by first considering usual geodesic
normal coordinates $x^\al$ on $\Si$ with respect to  
$\{e_\al\}$, extending the $e_{\al'}$ to $\Si$ to be parallel along radial
geodesics, and then taking $(x,u)$ to be the associated Fermi coordinates. 

Write $z=(x,u)$, so that $z^\al=x^\al$, $1\leq \al\leq k$, and
$z^{\al'}= u^{\al'}$, $k+1\leq \al'\leq n$.  Let lowercase Latin indices
$i$, $j$ run between $1$ and $n$.  In these coordinates the metric can be  
written 
\begin{equation}\label{gcoords}
  g=g_{ij}dz^idz^j =  g_{\al\be}dx^\al dx^\be +2g_{\al\al'}dx^\al
du^{\al'}+g_{\al'\be'}du^{\al'}du^{\be'}.
\end{equation}
Our main result is the following.
\begin{theorem}\label{main}
For each $i$, $j$ and list $K$ of indices, the derivative
$\pa_z^K g_{ij}(\sp)$ in submanifold geodesic normal coordinates can be    
expressed as a universal polynomial in the components at $\sp$ of the
curvature tensor of $g$ and its 
iterated covariant derivatives and the second fundamental form of $\Si$ and
its iterated covariant derivatives.  
\end{theorem}

Natural tensors on Riemannian manifolds are defined as  
isometry-invariant assignments of a tensor to each Riemannian manifold, 
subject to a polynomial regularity condition (see \cite{ABP}, \cite{E}). 
The analog of Theorem~\ref{main} for Riemannian metrics enables one to show  
that a natural tensor can be written as a polynomial in the curvature
tensor and its covariant derivatives which is orthogonally invariant when 
viewed as an algebraic object on which $O(n)$ acts, and 
Weyl's classical invariant theory shows that these are linear combinations
of contractions of tensor products.  As an application of
Theorem~\ref{main}, we formulate a definition of a natural tensor for
submanifolds and prove that any such tensor can be    
written as a linear combination of partial contractions of tensor products 
of the background curvature tensor and its 
covariant derivatives and the second fundamental form and its covariant
derivatives.  For simplicity we consider only covariant
tensors; the general case reduces to this upon lowering all indices. 

If $\Si^k\subset M^n$ is a submanifold, a general coordinate system
$(x,u)=(x^1,\cdots x^k,u^{k+1},\cdots,u^{n})$ near a point of $\Si$ is
called {\it adapted} if $\Si = \{u=0\}$.  Write $z=(x,u)$ as above.  In the 
following definition, the type $((p,q),(p',q'))$ of $\Si$ is fixed.  Note
that the type determines the signature $(p+p',q+q')$ of $g$, the dimension
$k=p+q$ of $\Si$, and the dimension $n=p+q+p'+q'$ of $M$.  

\begin{definition}\label{naturaltensors}
A natural tensor on submanifolds of type $((p,q),(p',q'))$, 
or a {\bf natural submanifold tensor}, 
is an assignment, to each embedded submanifold $\Si$ of type
$((p,q),(p',q'))$ of a pseudo-Riemannian manifold $(M,g)$, of a tensor
\[
T\in \Gamma\big((T^*\Si)^{\otimes r}\otimes (N^*\Si)^{\otimes s}\big)
\]
for some integers $r$, $s\geq 0$, such that the following two 
conditions hold:
\begin{enumerate}
\item
  If $\widetilde{\Si}\subset (\widetilde{M},\widetilde{g})$ and
    $\varphi: (M,g)\rightarrow (\widetilde{M},\widetilde{g})$ is an 
    isometry for which $\varphi(\Si) =\widetilde{\Si}$, then
    $\varphi^*\widetilde{T}=T$.  
\item
There are polynomials $T_{\cI\cJ}$ so that in any adapted coordinates
$z=(x,u)$, $T$ is given by  
\[
T=T_{\cI\cJ}\Big(g^{\al\be},g^{\al'\be'},\pa_z^Kg_{ij}\Big)
dx^\cI\otimes du^\cJ.  
\]
Here $\cI$ is a list of $r$ indices between $1$ and $k$ and $\cJ$ is a list
of $s$ indices between $k+1$ and $n$.  The argument $\pa_z^Kg_{ij}$ denotes 
all derivatives of all $g_{ij}$ of orders up to $N$ for some $N$, except  
that the variables $\pa_x^I g_{\al\al'}$ do not appear (since these vanish
in adapted coordinates).  The tensor $T$ and the $g^{\al\be}$,
$g^{\al'\be'}$, $\pa_z^Kg_{ij}$ are evaluated at $(x,0)$.   
\end{enumerate}
\end{definition}

\noindent
To clarify, $T_{\cI\cJ}$ is a polynomial function on the vector space in
which the inverse metric and the metric and its derivatives   
take values in local coordinates, taking into account the symmetry in the
metric and partial derivative indices and that the $\pa_x^I g_{\al\al'}$
vanish.  

We denote by $\na$ the Levi-Civita connection of $g$ and by $\nb$ the
induced connections on $T\Si$ and $N\Si$.  Let
$\Rm\in \Gamma(T^*M^{\otimes 4})$ denote the Riemann curvature tensor of
$g$ and let $L\in \Gamma(S^2T^*\Si\otimes N\Si)$ denote the second 
fundamental form of $\Si$, defined by  
$L(X,Y) = (\na_XY)^\perp$ for $X$,$Y\in T\Si$.
In Section~\ref{tensors} we prove 

\begin{theorem}\label{contractions}
Every natural submanifold tensor is an $\R$-linear combination of partial
contractions of tensors
\begin{equation}\label{tensor product}
\pi_1(\na^{M_1}\Rm)\otimes \cdots \otimes \pi_p(\na^{M_p}\Rm)\otimes 
\nb^{N_1}L\otimes\cdots\otimes\nb^{N_q}L\otimes \pi(g^{\otimes P}).
\end{equation}
Here $M_j$, $N_j$, and $P$ denote powers, and $\pi$ and $\pi_j$ denote
restriction to $\Si$ followed by projection to either $T\Si$ or $N\Si$ 
in each index.  This tensor is viewed as covariant in all indices (i.e.,
the $N\Si$ index on $L$ is lowered), and the 
contractions are taken with respect to the metrics induced by $g$ on
$T^*\Si$ and $N^*\Si$, for some partial pairings of tangential and normal
indices.     
\end{theorem}

Our interest in Theorem~\ref{contractions} arose in connection with the
papers \cite{CGK} and \cite{CGKTW}.  The main result of \cite{CGK} is a
construction of GJMS-type operators and $Q$-curvatures for submanifolds of
a Riemannian manifold.  Definitions similar to
Definition~\ref{naturaltensors} are given there for natural submanifold  
scalar differential operators and natural submanifold scalars,
and it is shown that the submanifold GJMS operators and $Q$-curvatures are 
natural in that sense.  However, it is not evident from the construction  
that they can be expressed in terms of curvature and second fundamental
form.  The submanifold $Q$-curvature plays an important role in 
\cite{CGKTW}, where such a description is needed.  We attempted  
unsuccessfully to find a version of Theorem~\ref{contractions} in the 
literature, hence the present paper.    

In Section \ref{normalcoordinates} we construct the submanifold geodesic
normal coordinates and prove Proposition~\ref{conditions}, which
characterizes them in terms of conditions on the metric.  In 
Section \ref{curvature} we prove Theorem~\ref{main} and also
Proposition~\ref{linear terms}, which 
identifies the linear 
terms in curvature and second fundamental form for the Taylor coefficients
of the metric in submanifold geodesic normal coordinates.  In
Section \ref{tensors} we prove Theorem~\ref{contractions} and   
indicate how this result can be extended to a characterization of natural 
differential operators.

\bigskip
\noindent
{\it Acknowledgements.}  This project is an outgrowth of the August 2022   
workshop ``Partial differential equations and conformal geometry'' at the
American Institute of Mathematics.  The authors are
grateful to AIM for its support and for providing a structure and
environment conducive to fruitful research.  We would like  
to thank Jeffrey Case and Jie Qing for helpful conversations.

\section{Submanifold Geodesic Normal Coordinates}\label{normalcoordinates}

Let $\Si \subset (M,g)$ be a submanifold of type $((p,q),(p',q'))$, with
inclusion $i$.  Let $\sp\in \Si$ and let $\{e_{\al}\}_{1\leq \al\leq k}$, 
$\{e_{\al'}\}_{k+1\leq \al'\leq n}$ be frames for 
$T_\sp\Si$, $N_\sp\Si$, resp., satisfying $g(e_{\al},e_{\be})=h_{\al\be}$, 
$g(e_{\al'},e_{\be'})=h_{\al'\be'}$.

Begin by constructing usual geodesic normal coordinates for $i^*g$ 
on $\Si$ near $\sp$.  Thus 
\[
\Si\ni \exp_\sp^{i^*g}(x^\al e_\al)\rightarrow x=(x^1,\cdots,x^k)\in \R^k. 
\]
In these coordinates, the curve $\gamma_x(t):=tx$ is
a geodesic for $i^*g$ for each $x\in \R^k$ near $0$.  Extend each
$e_{\al'}$ to a section of $N\Si$ in a  
neighborhood of $\sp$ in $\Si$ by requiring that it be parallel along each
$\gamma_x$ with respect to the induced connection on $N\Si$:
\[
\nb_{\dot{\gamma}_x} e_{\al'} = 0. 
\]
Then construct Fermi coordinates $u^{\al'}$ relative to $\{e_{\al'}\}$ and  
extend the $x^{\al}$ to be constant along the normal geodesics to obtain
the full submanifold geodesic normal coordinate system:
\[
M\ni \exp^g_{\exp_\sp^{i^*g}(x^\al e_\al)}(u^{\al'}e_{\al'})\rightarrow
(x,u)=(x^1,\cdots,x^k,u^{k+1},\cdots,u^{n}) \in
\R^k\times \R^{n-k}.
\]
These coordinates are characterized by $\pa_\al(\sp)=e_\al$, 
$\pa_{\al'}(\sp)=e_{\al'}$, and the following properties:
\begin{itemize}
\item[(I)]
  $\Si = \{u=0\}$ and $\sp=\{x=0, u=0\}$.
\item [(II)]
  The curve $\gamma_x(t)=(tx,0)$ is a geodesic for $i^*g$ for each
  $x\in \R^k$ near $0$.  
\item[(III)]
  $\pa_{\al'}|_\Si \in \Gamma(N\Si)$ and $\pa_{\al'}|_\Si$ is parallel
  with respect to $\nb$ along each $\gamma_x$.
\item[(IV)]
  The curve $\si_{x,u}(t)= (x,tu)$ is a geodesic for $g$ for each 
  $(x,u)\in \R^k\times \R^{n-k}$ near $(0,0)$. 
\end{itemize}

\begin{proposition}\label{conditions}
Let $(x,u)$ be a coordinate system near $\sp$ satisfying (I).  
Then $(x,u)$ is the submanifold geodesic normal coordinate system
corresponding to the frames $\pa_\al(\sp)=e_\al$,
$\pa_{\al'}(\sp)=e_{\al'}$, if and only if the following hold:
\begin{itemize}
\item[(A)]
  $g_{\al\be}(x,0)x^\be = h_{\al\be}x^\be$
\item[(B)]
  $g_{\al'\be'}(x,u)u^{\be'} = h_{\al'\be'}u^{\be'}$
\item[(C)]
  $g_{\al\al',\be'}(x,0)x^\al  = 0$
\item[(D)] 
  $g_{\al\al'}(x,u)u^{\al'}  = 0$
\end{itemize}
\end{proposition}
\begin{proof}
First we reformulate conditions (II)--(IV).
  
It is a classical fact that coordinates $x^\al$ are geodesic normal 
coordinates for a metric $g_{\al\be}(x)$ if and only if 
$g_{\al\be}(x)x^\be = h_{\al\be}x^\be$ (see, for example, Theorem 2.3 of
\cite{E} for a proof). So (II)$\iff$(A).

Condition (III) says $g_{\al\al'}(x,0) =0$ and
$\Gamma_{\al\al'}^{\be'}(x,0)x^\al=0$.  The latter is equivalent to
\[
g_{\al'\be',\al}(x,0)x^\al =
2g_{\al[\al',\be']}(x,0)x^\al,
\]
which is equivalent to
\[
g_{\al'\be',\al}(x,0)x^\al =0,\qquad g_{\al[\al',\be']}(x,0)x^\al=0,
\]
since the first is symmetric in $\al'\be'$ and the second is skew. So (III) 
is equivalent to
\begin{equation}\label{3equiv}
g_{\al\al'}(x,0)=0,\qquad g_{\al'\be',\al}(x,0)x^\al =0,\qquad
g_{\al[\al',\be']}(x,0)x^\al=0. 
\end{equation}

Condition (IV) is equivalent to the two conditions
\[
\Gamma_{\al'\be'}^{\ga'}(x,u)u^{\al'}u^{\be'}=0,\qquad
\Gamma_{\al'\be'}^{\al}(x,u)u^{\al'}u^{\be'}=0. 
\]
The first condition is the statement that for each $x$, $u^{\al'}$ is
a geodesic normal coordinate system for the metric 
$g_{\al'\be'}(x,u)du^{\al'}du^{\be'}$.  This is equivalent to (B) by the
classical fact quoted above.  The second condition is equivalent to 
\begin{equation}\label{IVpart}
g_{\al'\be',\al}(x,u)u^{\al'}u^{\be'} =
2g_{\al\al',\be'}(x,u)u^{\al'}u^{\be'}. 
\end{equation}
Note that applying $\pa_\al$ to (B) gives
$g_{\al'\be',\al}(x,u)u^{\be'} =0$, which implies that the left-hand side
of \eqref{IVpart} vanishes.  So in the presence of (B), \eqref{IVpart} is
equivalent to
\begin{equation}\label{IVpart2}
g_{\al\al',\be'}(x,u)u^{\al'}u^{\be'}=0.  
\end{equation}
Thus (IV) is equivalent to (B) and \eqref{IVpart2}.  

In order to show that (A)--(D)$\Longrightarrow$(II)--(IV), we only need to 
prove \eqref{3equiv} and \eqref{IVpart2}.

First we show \eqref{3equiv}.  Applying $\pa_{\al'}$ to (D) at $u=0$ 
gives $g_{\al\al'}(x,0)=0$, which is the first equation of \eqref{3equiv}.  
Applying $\pa_{\be'}$ to (B) at $u=0$ gives $g_{\al'\be'}(x,0)=h_{\al'\be'}$. So 
$g_{\al'\be',\al}(x,0)=0$, which implies the second equation of
\eqref{3equiv}.  Clearly (C) implies the third equation of 
\eqref{3equiv}.

In order to prove \eqref{IVpart2}, first apply $\pa_{\be'}$ to (D) to
obtain   
\[
g_{\al\al',\be'}(x,u)u^{\al'}+g_{\al\be'}(x,u)=0. 
\]
Now contract against $u^{\be'}$ and use (D) on the second term.  

We now show that (II)--(IV)$\Longrightarrow$ (A)--(D).  We have 
seen that (II)$\iff$(A) and (IV)$\iff$ (B) and \eqref{IVpart2}.  In
particular, it only remains to prove (C) and (D).

First we prove (C).  Applying $\pa^2_{\al'\be'}$ to \eqref{IVpart2} at
$u=0$ gives $g_{\al(\al',\be')}(x,0)=0$, so certainly  
$g_{\al(\al',\be')}(x,0)x^\al=0$.  Combined with the third equation of
\eqref{3equiv}, this gives $g_{\al\al',\be'}(x,0)x^\al=0$, which is (C).

Finally we prove (D).  Set $F_\al(x,u):=g_{\al\al'}(x,u)u^{\al'}$.  Then  
\[
u^{\be'}\pa_{\be'}F_\al = g_{\al\al',\be'}u^{\al'}u^{\be'}+g_{\al\be'}u^{\be'}
= F_\al,
\]
where we have used \eqref{IVpart2}.  So $F_\al$ is homogeneous of degree 1
in $u$.  But Condition (III) implies $g_{\al\al'}(x,0)=0$, so 
$F_\al=O(|u|^2)$.  Hence $F_\al = 0$, which is (D). 
\end{proof}

\section{Proof of Theorem~\ref{main}}\label{curvature}
There are several ways to prove the analog of Theorem~\ref{main} for 
Riemannian metrics.  The following proof generalizes the proof in
\cite{ABP}.  

\bigskip
\noindent
{\it Proof of Theorem~\ref{main}}.
For $1\leq \al\leq k$, define a section $\th^\al$ of $T^*\Si$ by parallel 
transport in $T^*\Si$ of $\big(dx^{\al}|_{T\Si}\big)(\sp)$ 
along radial geodesics in $\Si$.  Extend $\th^\al$ to a section of
$T^*M|_{\Si}$ by requiring that it vanish on $N\Si$.  Similarly, for
$k+1\leq \al'\leq n$, define a
section $\th^{\al'}$ of $N^*\Si$ by parallel transport in $N^*\Si$ of
$\big(du^{\al'}|_{N\Si}\big)(\sp)$ along radial 
geodesics in $\Si$, and extend $\th^{\al'}$ to a section of $T^*M|_\Si$ by
requiring that it vanish on $T\Si$.  Observe that
$\th^{\al'}=du^{\al'}$ on $\Si$.  Extend $\th^\al$, $\th^{\al'}$ to  
1-forms in $M$ near $\sp$ by parallel transport along normal geodesics.  

Parallel translation shows that $g|_{T\Si} = h_{\al\be}\th^\al \th^\be$ and 
$g|_{N\Si}= h_{\al'\be'}\th^{\al'}\th^{\be'}$.  Since
$g|_\Si=g|_{T\Si}+g|_{N\Si}$, it follows that
$g|_\Si=h_{ij}\th^i\th^j$.  
Parallel translation along normal geodesics now implies that
\begin{equation}\label{gframe}
g=h_{ij}\th^i\th^j
\end{equation}
in a neighborhood in $M$ of $\sp$.

Write
\begin{equation}\label{thetaa}
\th^i = a^i{}_j\, dz^j.
\end{equation}
Substituting into \eqref{gframe} and comparing with
\eqref{gcoords} shows that
\begin{equation}\label{ga}
  g_{ij}= h_{kl}a^k{}_ia^l{}_j.
\end{equation}
Note that  
\begin{equation}\label{aSigma}
a^\al{}_{\be'}|_\Si = 0, \qquad a^{\al'}{}_{\be}|_\Si = 0, \qquad
a^{\al'}{}_{\be'}|_\Si=\de^{\al'}{}_{\be'}.
\end{equation}

The connection 1-forms $\om^i{}_j$ for the frame $\{\th^i\}$ are defined by
\[
\na \th^i = \om^i{}_j\otimes\th^j,
\]
and satisfy
\begin{equation}\label{skew}
h_{ik}\om^k{}_j+h_{jk}\om^k{}_i=0.
\end{equation}
The structure equations are
\begin{equation}\label{structure}
d\th^i = \om^i{}_j\wedge \th^j,\qquad
d\om^i{}_j -\om^i{}_k\wedge \om^k{}_j = \Om^i{}_j, 
\end{equation}
where $\Om^i{}_j$ are the curvature 2-forms for the frame $\{\th^i\}$.
These are given by
\[
\Om^i{}_j=
-\tfrac12 \Rt^i{}_{jkl}\th^k\wedge \th^l
= -\tfrac12 a^i{}_r(a^{-1})^s{}_j R^r{}_{skl}  dz^k\wedge dz^l,
\]
where 
$\Rt^i{}_{jkl}$ are the components of the curvature tensor in the
frame $\{\th^i\}$ and $R^i{}_{jkl}$ are the components in the frame 
$\{dz^i\}$.  The components $\Lt_{\al\be}^{\al'}$ of $L$ 
relative to the frame $\{\th^\al\}$ and the components
$L_{\al\be}^{\al'}$ relative to $\{dx^\al\}$ satisfy
\begin{equation}\label{L}
\Lt_{\ga\be}^{\al'}\,a^\ga{}_\al  = -\om^{\al'}{}_\be(\pa_\al)
= L_{\al\ga}^{\al'}(a^{-1})^\ga{}_\be.
\end{equation}
Note that $(a^{-1})^\al{}_\be$ is unambiguous on $\Si$ because of
\eqref{aSigma}.

Set $\cU = u^{\al'}\pa_{\al'}$.  The $\th^i$ are parallel along the curves
$\si_{x,u}$, so
\[
\om^i{}_j(\cU)=0 \quad \text{on  } M.
\]
The tangent vectors at $t=0$ to the $\si_{x,u}$ span all of $N\Si$.
Therefore   
\begin{equation}\label{normalvanishing}
\om^i{}_j(N\Si)=0 \quad \text{on  } \Si. 
\end{equation}
The curves $\si_{x,u}$ are
geodesics satisfying $\dot{\si}_{x,u}=\cU$, so also $\cU$ is parallel along 
$\si_{x,u}$.  Hence $\th^i(\cU)$ is constant on $\si_{x,u}$.  Evaluating at
$t=0$ shows that
\begin{equation}\label{Ucontract}
\th^{\al}(\cU)=0,\qquad \th^{\al'}(\cU)=u^{\al'} \quad \text{on  } M.
\end{equation}

Similarly, set $\cX = x^\al \pa_\al$.  The $\th^\al|_{T\Si}$ and
$\th^{\al'}|_{N\Si}$ are parallel along the curves $\ga_x$, so 
\begin{equation}\label{Xvanish}
\om^\al{}_\be(\cX)=0,\qquad \om^{\al'}{}_{\be'}(\cX)=0 \quad \text{on  } \Si.
\end{equation}

We derive some identities involving the $a^i{}_j$ which we will use to
prove the theorem.  Shorten the notation for Lie derivatives by writing
just $\cU$ for $\cL_{\cU}$.  Begin by calculating
\[
\cU \th^i = \cU \into (\om^i{}_j\wedge \th^j) + d (\th^i(\cU))
= -u^{\ga'}\om^i{}_{\ga'} + d (\th^i(\cU)).  
\]
Next apply $|u|\circ \cU \circ |u|^{-1}$ to both sides, where $|\cdot|$
denotes the Euclidean norm on $\R^{n-k}$.  Recall Euler's
relation:  $\cU \eta = \lambda \eta$ if $\eta$ is a differential form
homogeneous of degree $\lambda$ in $u$.  Note that  
$|u|^{-1} d (\th^i(\cU))$ is homogeneous of degree 0 in $u$ by 
\eqref{Ucontract}.  Therefore
\begin{equation}\label{firstway}
|u|\circ \cU \circ |u|^{-1}(\cU \th^i)
= -u^{\ga'}\cU\into d\om^i{}_{\ga'} = -u^{\ga'}\cU\into \Om^i{}_{\ga'} 
=a^i{}_r (a^{-1})^s{}_{\ga'} R^r{}_{s\al'l} u^{\al'}u^{\ga'}dz^l.  
\end{equation}
Now do a different calculation of the left-hand side directly from
\eqref{thetaa}: 
\[
\cU\th^i =
\cU(a^i{}_j) dz^j + a^i{}_{\be'} du^{\be'}.
\]
So
\[
|u|\circ \cU \circ |u|^{-1}(\cU \th^i)
=\big(\cU^2a^i{}_\be-\cU a^i{}_\be\big)dx^\be
+\big(\cU^2a^i{}_{\be'}+\cU a^i{}_{\be'}\big)du^{\be'}.
\]
It follows that
\begin{equation}\label{Uderivs1}
(\cU^2-\cU)a^i{}_\be
  = a^i{}_r (a^{-1})^s{}_{\ga'} R^r{}_{s\al'\be} u^{\al'}u^{\ga'}
\end{equation}
\begin{equation}\label{Uderivs2}
(\cU^2+\cU)a^i{}_{\be'}
 = a^i{}_r (a^{-1})^s{}_{\ga'} R^r{}_{s\al'\be'} u^{\al'}u^{\ga'}.     
\end{equation}

Next carry out the same calculation as above on 
$\Si$ with respect to $\cX$ and the metric $i^*g$.  
One obtains
\begin{equation}\label{Xderivs}
(\cX^2+\cX)a^\al{}_{\be}
  = a^\al{}_\ga (a^{-1})^\de{}_{\rho} \Rb^\ga{}_{\de\si\be}
  x^{\si}x^{\rho}\quad \text{on  } \Si,
\end{equation}
where $\Rb$ denotes the curvature tensor of $i^*g$.  

We also need to know about the first normal derivatives of $a^\al{}_\be$
and $a^{\al'}{}_{\be}$.  {From} \eqref{thetaa}, \eqref{aSigma}, we have on 
$\Si$: 
\[
d\theta^i = d_ua^i{}_\be \wedge dx^\be
\mod \{dx^\ga\wedge dx^\de, du^{\al'}\wedge du^{\be'}\},
\]
where $d_u$ denotes the exterior derivative in the $u$ variables.  
Taking $i=\al$, comparing with \eqref{structure}, and recalling 
\eqref{normalvanishing}, 
\eqref{skew}, \eqref{L} gives 
\begin{equation}\label{aL}
\pa_{\al'}a^\al{}_\be = -\om^\al{}_{\al'}(\pa_\be)
=-h^{\al\ga}h_{\al'\be'}L_{\be\de}^{\be'}(a^{-1})^\de{}_\ga \quad \text{on  }\Si.
\end{equation}
Next, taking $i=\al'$, comparing with \eqref{structure} and using 
\eqref{normalvanishing} gives 
$\pa_{\be'}a^{\al'}{}_\be = -\om^{\al'}{}_{\be'}(\pa_\be)$.  Thus 
\[
\cX \big(\pa_{\be'}a^{\al'}{}_\be \big)
=-(\cX\om^{\al'}{}_{\be'})(\pa_\be)-\om^{\al'}{}_{\be'}(\cX\pa_\be). 
\]
Since $\om^{\al'}{}_{\be'}(\cX\pa_\be)=
-\om^{\al'}{}_{\be'}(\pa_\be) =\pa_{\be'}a^{\al'}{}_\be$, we obtain
\[
(\cX+1) \big(\pa_{\be'}a^{\al'}{}_\be \big) =
-(\cX\om^{\al'}{}_{\be'})(\pa_\be).
\]
Using \eqref{Xvanish}, \eqref{structure}, this becomes 
\begin{equation}\label{X+1}
\begin{split}
(\cX+1) \big(\pa_{\be'}a^{\al'}{}_\be \big) 
& =- (\cX\into d\om^{\al'}{}_{\be'})(\pa_\be)
=-\cX\into\big(\Om^{\al'}{}_{\be'}
+\om^{\al'}{}_\ga\wedge \om^\ga{}_{\be'}\big)(\pa_\be)\\
&=R^{\al'}{}_{\be'\al\be}x^\al
+h^{\ga\de}h_{\be'\ga'}\big(L^{\al'}_{\rho\al}L^{\ga'}_{\be\si}
-L^{\al'}_{\rho\be}L^{\ga'}_{\al\si}\big)(a^{-1})^\rho{}_\ga
(a^{-1})^\si{}_\de x^\al.
\end{split}
\end{equation}

We now prove by induction on $N$ the following statement:  For all $i$,
$j$, $I$ and for all $|J|\leq N$, each coordinate derivative  
\begin{equation}\label{deriva}
(\pa_x)^I(\pa_u)^Ja^i{}_j(0,0)
\end{equation}
can be expressed as a polynomial in the coordinate derivatives at
$\sp=(0,0)$ of the components in the frame $\{dz^i\}$ of the curvature
tensor of $g$ and 
the second fundamental form.  For $N=0$, this is clear from \eqref{aSigma}
for $a^\al{}_{\al'}$, $a^{\al'}{}_\al$, and $a^{\al'}{}_{\be'}$.  For
$a^\al{}_\be$ we proceed by induction on $M=|I|$.  The result is true  
for $M=0$ since $a^\al{}_\be(\sp)=\de^\al{}_\be$.  Denote by
$\widehat{a^\al{}_\be}(M)$ the homogeneous term of degree $M$ in the 
Taylor polynomial at $x=0$ of $a^\al{}_\be(x,0)$.  Take the
homogeneous term of degree $M$ in the Taylor polynomial of both sides of 
\eqref{Xderivs}.  The left-hand side gives
$(M^2+M)\widehat{a^\al{}_\be}(M)$.  Since $\Rb$ can be expressed in terms
of $R$ and $L$ by the Gauss curvature equation, the coefficients in the
term of degree $M$ on the right-hand side can each be written as a
polynomial in derivatives 
of components of $R$, $L$, and coefficients of $\widehat{a^\al{}_\be}(M')$
with $M'\leq M-2$.  The conclusion for $N=0$ and $M>0$ follows by
induction.  

For fixed $x$, denote by $\widehat{a^i{}_j}(x,N)$ the homogeneous term of 
degree $N$ in the Taylor polynomial in $u$ at $u=0$ of $a^i{}_j(x,u)$.
Denote by $\widehat{a^i{}_j}(M,N)$ the homogeneous term of degree $M$ in
$x$ and $N$ in $u$ at $(x,u)=(0,0)$.  
Equation \eqref{Uderivs2} shows that $\widehat{a^i{}_{\be'}}(x,1)=0$.
Differentiating in $x$ gives $\widehat{a^i{}_{\be'}}(M,1)=0$
for all $M\geq 0$.  This proves the induction statement for $N=1$ for 
$a^\al{}_{\be'}$ and $a^{\al'}{}_{\be'}$.  A different
argument is needed for $a^\al{}_{\be}$ and
$a^{\al'}{}_{\be}$ since $\cU^2-\cU$ annihilates 
$\widehat{a^i{}_\be}(x,1)$ in \eqref{Uderivs1}.  For
$a^\al{}_{\be}$, the statement for $N=1$ 
follows directly upon differentiating \eqref{aL} in $x$.  
For $a^{\al'}{}_{\be}$, take homogeneous terms in $x$ of both 
sides of \eqref{X+1}.
This completes the proof of the induction step for $N=1$.  
For $N>1$, the induction step is proved by first taking the homogeneous
term of degree $N$ in $u$ for fixed $x$ in \eqref{Uderivs1},
\eqref{Uderivs2} and then differentiating the result in $x$ at $x=0$. 

It now follows from \eqref{ga} that each derivative of $g_{ij}$ at $\sp$
can be expressed as a polynomial in the coordinate derivatives at $\sp$ of  
the components of the curvature tensor and 
the second fundamental form.  Each such coordinate derivative of curvature
or second fundamental form equals the corresponding covariant derivative
plus terms involving lower order derivatives of the metric.  So the
conclusion follows by induction on the order of differentiation of the
metric.  
\qed

\bigskip
The next proposition identifies the linear terms in curvature and second  
fundamental form in the $\pa_z^Kg_{ij}(\sp)$. 
We will not use this result, but it is illuminating and might be useful in
the future.  

Note that the homogeneous term of degree $M$ in $x$ and $N$ in $u$ in the
Taylor series at $(0,0)$ of a
function $f(x,u)$ on $\R^k\times \R^{n-k}$ is given by 
\[
\widehat{f}(M,N) = \frac{1}{M!N!}f_{\ga_1\cdots \ga_M\ga_1'\cdots
  \ga_N'}x^{\ga_1}\cdots x^{\ga_M}u^{\ga_1'}\cdots u^{\ga_N'},
\]
where the subscripts denote partial derivatives and are understood to be
evaluated at $(0,0)$.   In the following, all identities hold modulo   
terms of degree $2$ or higher in $R$, $L$, and their derivatives.  
We use $|\cdot|$ to indicate an index not included in the symmetrization
$(\cdots)$. 

\begin{proposition}\label{linear terms}  
Modulo terms of degree 2 and higher in $R$, $L$, and their derivatives, the 
derivatives of $g_{ij}$ at $\sp=(0,0)$ are given by:
\begin{equation}\label{linearg}
\begin{aligned}
g_{\al\be}&=h_{\al\be}&\\
g_{\al\be,\ga}&=0&\\
g_{\al\be,\ga_1\cdots \ga_M}&
=2\tfrac{M-1}{M+1}R_{\al(\ga_1\ga_2|\be|,\ga_3\cdots \ga_M)} \qquad &M\geq 2\\
g_{\al\be,\ga_1\cdots \ga_M\ga'}&=-2L_{\al\be\ga',\ga_1\cdots \ga_M} \qquad
&M\geq 0\\
g_{\al\be,\ga_1\cdots \ga_M\ga_1'\cdots \ga_N'}&=
2R_{\al(\ga_1'\ga_2'|\be|,\ga_3'\cdots \ga_N')\ga_1\cdots \ga_M} \qquad
&M\geq 0, \quad N\geq 2\\
g_{\al\be',\ga_1\cdots \ga_M}&=0 \qquad &M\geq 0\\
g_{\al\be',\ga'}&=0&\\
g_{\al\be',\ga_1\cdots \ga_M\ga'}&=
-\tfrac{M}{M+1}
R_{\be'\ga'\al(\ga_1,\ga_2\cdots \ga_M)} \qquad &M\geq 1\\
g_{\al\be',\ga_1\cdots \ga_M\ga_1'\cdots \ga_N'}&=2\tfrac{N}{N+1} 
R_{\al(\ga_1'\ga_2'|\be'|,\ga_3'\cdots \ga_N')\ga_1\cdots \ga_M} \qquad
&M\geq 0, \quad N\geq 2\\
g_{\al'\be'}&=h_{\al'\be'}&\\
g_{\al'\be',\ga_1\cdots \ga_M}&=0 \qquad &M\geq 1\\
g_{\al'\be',\ga_1\cdots \ga_M\ga'}&=0 \qquad &M\geq 0\\
g_{\al'\be',\ga_1\cdots \ga_M \ga_1'\cdots \ga_N'}&=2\tfrac{N-1}{N+1}
R_{\al'(\ga_1'\ga_2'|\be'|,\ga_3'\cdots \ga_N')\ga_1\cdots \ga_M} \qquad
&M\geq 0, \quad N\geq 2.
\end{aligned}
\end{equation}
\end{proposition}
\begin{proof}
First we identify the linear terms in the derivatives of the $a^i{}_j$,
organized by the order $N$ of differentiation in
$u$ as in the proof of Theorem~\ref{main}.

\bigskip\noindent
$N=0$:
\bigskip

\noindent
Differentiation of \eqref{aSigma} gives
\begin{equation}\label{1a}
\begin{aligned}
&a^\al{}_{\be',\ga_1\cdots \ga_M}=0\qquad M\geq 0,\\
&a^{\al'}{}_{\be,\ga_1\cdots   \ga_M}=0\qquad M\geq 0,\\
&a^{\al'}{}_{\be'}=\delta^{\al'}{}_{\be'}\\
&a^{\al'}{}_{\be',\ga_1\cdots \ga_M}=0\qquad M\geq 1.
\end{aligned}
\end{equation}
Equation~\eqref{Xderivs} shows that $\wh{a^\al{}_{\be}}(1,0)=0$ and
\[
(M^2+M)\wh{a^\al{}_{\be}}(M,0)= \wh{R^\al{}_{\rho\ga\be}}(M-2,0)x^\rho
x^\ga\qquad M\geq 2.
\]
So
\begin{equation}\label{2a}
a^{\al}{}_{\be} = \delta^{\al}{}_{\be},\qquad
a^{\al}{}_{\be,\ga} = 0,\qquad
a^{\al}{}_{\be,\ga_1\cdots \ga_M}=\tfrac{M-1}{M+1}
R^\al{}_{(\ga_1\ga_2|\be|,\ga_3\cdots \ga_M)} \qquad M\geq 2.
\end{equation}

\bigskip\noindent
$N=1$:
\bigskip

\noindent
Equation~\eqref{Uderivs2} shows that
\begin{equation}\label{3a}
a^i{}_{\be',\ga_1\cdots \ga_M \ga'}=0\qquad M\geq 0.
\end{equation}
Differentiating \eqref{aL} gives
\begin{equation}\label{4a}
a^\al{}_{\be,\ga_1\cdots \ga_M \ga'}=
-h^{\al\de}h_{\ga'\be'}L^{\be'}_{\be\de,\ga_1\cdots\ga_M}\qquad M\geq 0.
\end{equation}
Equation~\eqref{X+1} gives
\begin{equation}\label{5a}
a^{\al'}{}_{\be,\ga'} =0,\qquad
a^{\al'}{}_{\be,\ga_1\cdots \ga_M \ga'}=-\tfrac{M}{M+1}
R^{\al'}{}_{\ga'\be(\ga_1,\ga_2\cdots \ga_M)} \qquad M\geq 1.
\end{equation}

\bigskip\noindent
$N\geq 2$:
\bigskip

\noindent
Equation~\eqref{Uderivs1} gives
\begin{equation}\label{6a}
a^i{}_{\be,\ga_1\cdots \ga_M \ga_1'\cdots \ga_N'}=
R^i{}_{(\ga_1'\ga_2'|\be|,\ga_3'\cdots \ga_N')\ga_1\cdots \ga_M} \qquad M\geq 0,
\end{equation}
and \eqref{Uderivs2} gives
\begin{equation}\label{7a}
a^i{}_{\be',\ga_1\cdots \ga_M \ga_1'\cdots \ga_N'}=\tfrac{N-1}{N+1}
R^i{}_{(\ga_1'\ga_2'|\be'|,\ga_3'\cdots \ga_N')\ga_1\cdots \ga_M} \qquad M\geq 0.
\end{equation}

Now consider the linear term in a derivative $Dg_{ij}$, where $D$ denotes
some iterated derivative with respect to $x$ and $u$ evaluated at $(0,0)$.   
The iterated Leibnitz rule applied to \eqref{ga} gives a
sum of quadratic terms in derivatives of the $a^k{}_l$.  By the above, the
only derivatives of an $a^k{}_l$ that have a nonzero constant term when
viewed as a polynomial in $R$ and $L$ and their derivatives are the
undifferentiated  $a^\al{}_\be=\de^\al{}_\be$ and
$a^{\al'}{}_{\be'}=\de^{\al'}{}_{\be'}$. 
So in order to obtain a nonzero linear term in $Dg_{ij}$,
all of the derivatives must land on the same $a^k{}_l$.  It follows that
\begin{equation}\label{atog}
\begin{aligned}
Dg_{\al\be}&=2h_{\rho(\al}Da^\rho{}_{\be)}\\
Dg_{\al'\be'}&=2h_{\rho'(\al'}Da^{\rho'}{}_{\be')}\\ 
  Dg_{\al\be'}&=h_{\al\rho}Da^\rho{}_{\be'}
  +h_{\be'\rho'}Da^{\rho'}{}_{\al},
\end{aligned}
\end{equation}
modulo quadratic terms in $R$ and $L$ and their derivatives.  
Substituting \eqref{1a}--\eqref{7a} into \eqref{atog} and using
$h_{\al\be}$ and $h_{\al'\be'}$ and their inverses to lower and raise 
indices gives \eqref{linearg}.  
\end{proof}

\section{Natural Tensors}\label{tensors}
In this section we prove Theorem~\ref{contractions}.  We begin with an
algebraic lemma, the analog of which for $SO(n)$
appears as Theorem 8 in Section 8 of Chapter 5 of \cite{BFG}.    
The proof there is missing a step.  We are grateful to Charles
Fefferman for showing us how to complete the proof. 

Fix $((p,q),(p',q'))$ and set $k=p+q$, $n=p+q+p'+q'$.  For $r$, $s\geq 0$,
set $\T^{r,s} =(\R^k{}^*)^{\otimes r}\otimes (\R^{n-k}{}^*)^{\otimes s}$ and
recall that $H=O(p,q)\times O(p',q')$.  There is a natural action of $H$ on 
$\T^{r,s}$.   

\begin{lemma}\label{extension}
Set $\bT:=\T^{r_1,s_1}\oplus \cdots \oplus \T^{r_L,s_L}$ for some choice 
of powers $r_i$, $s_i$, $1\leq i\leq L$, and let $\cE\subset \bT$ be a  
non-empty $H$-invariant set.  If $P$ is a polynomial on $\bT$ whose
restriction to $\cE$ is $H$-invariant, then there is a 
polynomial $P'$ which is $H$-invariant on all of $\bT$ such that $P'=P$ on 
$\cE$.   
\end{lemma}

\noindent
In the definite case, $P'$ can be obtained by averaging $P$ over $H$.  The
proof below is valid also in the indefinite case.  Instead of integration,
it uses polarization and semisimplicity of $H$.   

\begin{proof}
Write elements of $\bT$ as $(T^1,\cdots,T^L)$ with $T^i\in \T^{r_i,s_i}$. 
Decompose $P$ into its fully homogeneous pieces:  write 
$P=\sum_{j=1}^MP_j$, where $P_j$ is homogeneous of degree $m_{ij}$ in $T^i$
for $1\leq i\leq L$.  Define
\[
\cT_j = \bigotimes_{i=1}^L \big(\T^{r_i,s_i}\big)^{\otimes m_{ij}},\quad 
1\leq j\leq M, 
\]
and set $\cT:=\bigoplus_{j=1}^M\cT_j$, with projections
$\pi_j:\cT\rightarrow \cT_j$. 
Define $\phi_j:\bT\rightarrow \cT_j$, $1\leq j \leq M$, by 
\[
\phi_j(T^1,\cdots,T^L)
=\underbrace{T^1\otimes\cdots\otimes T^1}_{m_{1j}\rm times}\otimes
\cdots\otimes
\underbrace{T^L\otimes\cdots\otimes T^L}_{m_{Lj}\rm times} 
\]
and define $\phi:\bT\rightarrow \cT$ by $\phi = (\phi_1,\cdots, \phi_M)$. 
Clearly the $\phi_j$ and $\phi$ are $H$-equivariant.
Polarization of $P_j$ produces a linear map 
$\ell_j:\cT_j\rightarrow \R$ so that $P_j=\ell_j\circ \phi_j$.    Define
$\ell:\cT\rightarrow \R$ by
$\ell = \sum_{j=1}^M \ell_j\circ\pi_j$.  
Then $P=\ell\circ\phi$.  

Set $\cE^+=\Span \phi(\cE)$.  Then $\cE^+$ is an $H$-invariant
subspace of $\cT$, and $\ell|_{\cE^+}$ is $H$-invariant since $P|_\cE$ is
$H$-invariant.  Since the Lie algebra of $H$ is semisimple
and $H$ has finitely many components, $\cE^+$ has an $H$-invariant
complement $\cE^-$ (see \cite[Th\'eor\`em 3b p. 85]{C}).
Define $\ell':\cT\rightarrow \R$ by $\ell'=\ell$ on 
$\cE^+$, $\ell'=0$ on $\cE^-$, and extend by linearity.  The polynomial
$P' = \ell'\circ \phi$ is then an $H$-invariant extension of $P$ to $\bT$.   
\end{proof} 

\bigskip
\noindent
Given Theorem~\ref{main}, the proof of Theorem~\ref{contractions} is
similar to that for natural tensors on Riemannian manifolds, as in
\cite{ABP}, for example.  

\bigskip\noindent
{\it Proof of Theorem~\ref{contractions}.}
First restrict attention to metrics in submanifold geodesic normal
coordinates about the origin.  The metric at the origin is determined:
$g_{ij}(0,0)=h_{ij}$.  Successive differentiation at the origin 
of (A)--(D) of Proposition~\ref{conditions} gives a set of linear 
relations on the derivatives $\pa_z^Kg_{ij}$, $|K|\geq 1$, which are 
necessary and sufficient for $g$ to be in submanifold geodesic normal
coordinates to infinite order.  Define the vector space
\[
\bG:= \{(\pa_z^Kg_{ij})_{|K|\geq 1}: \text{(A)--(D) hold to infinite order}\}
\]
and its finite-dimensional projection
$\bG_N:=\{(\pa_z^Kg_{ij})_{1\leq |K|\leq N+2}\}$
obtained by truncation at order $N+2$.    

The decomposition $T^*M|_\Si = T^*\Si\oplus N^*\Si$ induces a decomposition  
\begin{equation}\label{tensor decompose}
(T^*M|_\Si)^{\otimes k}=\bigoplus\big((T^*\Si)^{\otimes \ell}\otimes  
(N^*\Si)^{\otimes m}\big),    
\end{equation}
where the direct sum is over all ways of choosing either 
$T^*\Si$ or $N^*\Si$ for each of the $k$ slots.  For each choice, $\ell$ 
and $m$ denote the number of choices of $T^*\Si$ and $N^*\Si$, resp.   
We decompose each iterated covariant derivative
$\na^kR|_\Si\in (T^*M|_\Si)^{\otimes (k+4)}$ of the curvature tensor of $g$  
according to \eqref{tensor decompose}.  After lowering the $N\Si$ index,
the iterated covariant derivative 
$\nb^{k+1}L$ of the second fundamental form is already a section of 
$(T^*\Si)^{\otimes (k+3)}\otimes N^*\Si$.    
The tensors $\na^kR|_\Si$ and $\nb^{k+1}L$ depend only on derivatives of
$g$ of order $\leq k+2$.  Consequently, evaluation of 
\begin{equation}\label{RL list}
(R,\na R,\cdots, \na^NR,L,\nb L,\cdots, \nb^{N+1}L) 
\end{equation}
at the origin determines a polynomial map
\[
\cR:\bG_N\rightarrow \bigoplus \T^{\ell,m}=:\bR_N, 
\]
where now the direct sum is over all summands that occur 
when decomposing every $\na^kR|_\Si$ in \eqref{RL list}, together with the 
summands corresponding to the $\nb^kL$.  We suppress writing the implied
subscript $N$ on $\cR$.  

If $h\in H=O(p,q)\times O(p',q')$ and $g$ is in submanifold geodesic   
normal coordinates, then $h^*g$ is also in submanifold geodesic normal
coordinates.  This 
determines an action of $H$ on $\bG_N$, which is the direct sum of the
usual action on the 
tensors $\pa_z^Kg_{ij}$.  The map $\cR:\bG_N\rightarrow \bR_N$ is 
$H$-equivariant, and in particular its range $\cR(\bG_N)$ is $H$-invariant.    

Theorem~\ref{main} implies the existence of a polynomial map
$\cG:\bR_N\rightarrow \bG_N$ so that $\cG\circ \cR = Id$.   
The map $\cG|_{\cR(\bG_N)}$ is also $H$-equivariant since it is the inverse
of an equivariant map.  

Let now $T$ be a natural submanifold tensor as in
Definition~\ref{naturaltensors}, taking values in $(T^*\Si)^{\otimes
  r}\otimes (N^*\Si)^{\otimes s}$.   
Evaluating $T$ at the origin for a metric in submanifold normal coordinates
gives a tensor in $\T^{r,s}$, each of
whose components is a polynomial on $\bG_N$.  Set $\Tt:=T\circ \cG$.   
Then $\Tt$ is a polynomial map
\begin{equation}\label{T}
\Tt:\bR_N\rightarrow \T^{r,s}.  
\end{equation}
Isometry invariance of $T$ implies that $\Tt|_{\cR(\bG_N)}$ is
$H$-equivariant. 

Define a polynomial 
$P:\bR_N\oplus \T^{r,s} \rightarrow \R$ 
by
\[
P(t,S)=\langle \Tt(t),S\rangle, \qquad t\in \bR_N,\quad S\in \T^{r,s},
\]
where $\langle\cdot,\cdot\rangle$ denotes the quadratic form on $\T^{r,s}$
induced by $g$.  Then
$P|_{\cR(\bG_N)\times \T^{r,s}}$ is $H$-invariant.
Lemma~\ref{extension} with $\bT=\bR_N\oplus \T^{r,s}$ and
$\cE=\cR(\bG_N)\times \T^{r,s}$ implies that there is 
an $H$-invariant polynomial $P':\bR_N\oplus \T^{r,s} \rightarrow \R$ such 
that $P'=P$ on $\cR(\bG_N)\times \T^{r,s}$.

Weyl's classical invariant theory (\cite{W} and/or Section 8 of   
Chapter 5 of \cite{BFG}) shows that any $H$-invariant polynomial 
on $\oplus_i \T^{r_i,s_i}$ is a linear combination of complete
contractions, so $P'$ has this property.  Removing the last tensor $S$
shows that $\Tt$ is a linear combination of partial contractions of tensors
of the form \eqref{tensor product} for metrics in submanifold geodesic
normal coordinates.  So $T=\Tt\circ \cR$ is too.  

The result for general metrics and submanifolds follows by putting the 
metric into submanifold geodesic normal coordinates by a diffeomorphism,
since \eqref{tensor product} transforms tensorially.     
\qed

\bigskip
Natural submanifold differential operators on $\Si$ between functorial
subbundles of the bundles
$(T^*\Si)^{\otimes r}\otimes (N^*\Si)^{\otimes  s}$  
can be defined by analogy with Definition~\ref{naturaltensors} and  
the treatment in \cite{S}:  in local
coordinates, the coefficient of each partial derivative is required to  
depend polynomially on the inverse metric and the metric and its
derivatives. 
Theorem~\ref{contractions} implies a similar characterization for 
natural submanifold differential operators, since a differential operator  
$D:\Gamma(E)\rightarrow \Gamma(F)$ between bundles $E$ and $F$ can be
uniquely written as  
\[
D = \sum_{k=0}^m a_k \Sym(\nb^k)
\]
where $\Sym(\nb^k)$ denotes the symmetrization of the $k^{\text{th}}$ iterated
covariant derivative and
$a_k\in \Gamma(S^kT\Si\otimes E^*\otimes F)$.  If $D$ is a natural
submanifold differential operator, then each $a_k$ is a natural
submanifold tensor to which Theorem~\ref{contractions} applies.

\end{document}